\documentclass[11pt]{amsart}
\usepackage{amscd,amsmath,amssymb,amsfonts}
\usepackage[all,cmtip]{xy}
\usepackage{hyperref}
\usepackage{stmaryrd}

\theoremstyle{plain}
\newtheorem{thm}{Theorem}

\newtheorem{lem}[thm]{Lemma}
\newtheorem{cor}[thm]{Corollary}
\newtheorem{prop}[thm]{Proposition}

\newtheorem{ex}[thm]{Example}

\theoremstyle{definition}

\newtheorem{question}[thm]{Question}

\newtheorem{claim}[thm]{Claim}
\newtheorem{rmk}[thm]{Remark}

\newcommand{\Pic}{{\rm Pic}}

\DeclareMathOperator{\Spec}{Spec}

\newcommand{\sC}{{\mathcal C}}

\newcommand{\sF}{{\mathcal F}}

\newcommand{\sM}{{\mathcal M}}
\newcommand{\sN}{{\mathcal N}}
\newcommand{\sO}{{\mathcal O}}

\newcommand{\sU}{{\mathcal U}}

\newcommand{\sX}{{\mathcal X}}

\newcommand{\B}{{\mathbb B}}

\newcommand{\Sp}{{\rm Sp}}

\newcommand{\prolim}[1]{\operatorname{``}\lim\limits_{#1}   \operatorname{''}}

\newcommand{\Ve}{\operatorname{Vec}}

\DeclareMathOperator{\coker}{coker}

\newcommand{\checkH}{H}   

\begin{document}

\title[Non-archimedean Bass--Quillen]{Towards a non-archimedean analytic analog of the Bass--Quillen conjecture}

\author{Moritz Kerz}
\address{Fakult\"at f\"ur Mathematik, Universit\"at Regensburg, 93040 Regensburg, Germany}
\email{moritz.kerz@ur.de}

\author{Shuji Saito}
\address{Graduate School of Mathematical Sciences, 
University of Tokyo,
3-8-1 Komaba, Tokyo, 153-8914,
Japan
}
\email{sshuji@msb.biglobe.ne.jp}

\author{Georg Tamme}
\address{Fakult\"at f\"ur Mathematik, Universit\"at Regensburg, 93040 Regensburg, Germany}
\email{georg.tamme@ur.de}

\thanks{The authors are supported by the DFG through CRC 1085 \emph{Higher Invariants} (Universit\"at Regensburg)}

\begin{abstract}
We suggest an analog of the Bass--Quillen conjecture for smooth affinoid algebras over a complete non-archimedean field.
We prove this in the rank-1 case, i.e.~for the Picard group. For complete discretely valued fields 
and regular affinoid algebras that admit a regular model (automatic if the residue characteristic is zero)
 we prove a similar statement for the Grothendieck group of vector bundles $K_{0}$.
\end{abstract}

\maketitle

\section*{Introduction}

For a ring $A$ let us denote by $\Ve_r(A)$ the set of isomorphism classes of finitely generated projective
modules of rank $r$.
The Bass--Quillen conjecture predicts that for a regular noetherian ring $A$ the inclusion
into the polynomial ring $A[t_1,\ldots, t_n]$ 
induces a bijection
\[
\Ve_r(A)\xrightarrow{\sim} \Ve_r(A[t_1,\ldots, t_n])
\]
for all $n,r\ge 0$. Based on the work of Quillen and Suslin on Serre's problem the
conjecture has been shown in case $A$ is a smooth algebra over a field~\cite{Li}.

In this note we discuss a potential extension of this
conjecture to affinoid algebras in the sense of Tate. Let $K$ be a field which is complete
with respect to a non-trivial non-archimedean absolute value and let $A/K$ be a smooth affinoid
algebra. In rigid geometry a building block is the ring of power series converging on the closed unit
disc   
\[
A\langle t_1 , \ldots , t_n \rangle= \{f=\sum_{\underline k} c_{\underline k} t^{\underline k} \in A\llbracket t_1,\ldots , t_n \rrbracket \, |
\, c_{\underline k}\xrightarrow{|\underline k|\to \infty} 0 \},
\]
which serves as a replacement for the polynomial ring in algebra. 

Using these convergent power series the following positive result in analogy with Serre's problem is
obtained in \cite{L}.

\begin{ex}[L\"utkebohmert]
All finitely generated projective modules over $K\langle t_1,\ldots, t_n\rangle$ are free.
\end{ex}

Unfortunately, over more general smooth affinoid algebras one has the following negative
example \cite[4.2]{Ge}.

\begin{ex}[Gerritzen]
Assume the ring of integers $K^\circ$ of $K$ is a discrete valuation ring with prime element $\pi$.
For the smooth affinoid $K$-algebra $A=K\langle t_1, t_2 \rangle/(t_1^2-t_2^3-\pi)$ the map 
\[
\Pic (A) \to \Pic (A\langle t \rangle )
\]
is not bijective.
\end{ex}

This example shows that for our purpose the ring of convergent power series $A\langle
t\rangle$ is not entirely appropriate. 
Let $\pi\in K\setminus \{ 0\}$ be an element with $|\pi|<1$.
As an improved non-archimedean analytic replacement for the polynomial ring over $A$ we are going to
use the pro-system of affinoid algebras $\prolim{t\mapsto \pi t}   A\langle
t \rangle$. It represents an affinoid approximation of the non-quasi-compact rigid analytic space $(\mathbb
A^1_A)^{\rm an}$ since \[\lim_{t\mapsto \pi t} A\langle
t \rangle= H^0((\mathbb
A^1_A)^{\rm an},\sO).\]
Note that the latter non-affinoid $K$-algebra is harder to control, compare
\cite[Ch.~5]{Gr} and \cite{B4}.

As a non-archimedean analytic analog of the Bass--Quillen conjecture one might ask:

\begin{question}\label{qu.main}
Is the map
\[
\Ve_r(A) \to \prolim{{t\mapsto \pi t} } \Ve_r(A\langle t \rangle)
\]
a pro-isomorphism for $A/K$ a smooth affinoid algebra?
\end{question}

We give a positive answer for $r=1$.

\begin{thm}\label{thm1}
For $A/K$ a smooth affinoid algebra the map
\[
\Pic(A) \to \prolim{t\mapsto \pi t}  \Pic(A\langle t \rangle)
\]
is an isomorphism of pro-abelian groups.
\end{thm}
This  is  stronger than the  statement that $\Pic(A) \to \lim_{t\mapsto \pi t}  \Pic(A\langle t \rangle)$ is an isomorphism.
The latter has the following consequence, which we will prove in Section~\ref{sec.pic}:
\begin{cor}\label{cor.A1invariance}
For $A/K$ a smooth affinoid algebra the map
\[
\Pic(A) \to   \Pic((\mathbb{A}^{1}_{A})^{\rm an})
\]
is an isomorphism.
\end{cor}

The Picard group $\Pic(A)$ of an affinoid algebra $A$ is isomorphic to the 
cohomology group $H^1(\Sp(A), \sO^*)$.

In case the residue field of $K$ has characteristic zero, one has the exponential isomorphism $\exp
: \sO(1)\xrightarrow{\sim} \sO^*(1) $, where $\sO(1)\subset \sO$ is the subsheaf of rigid
analytic functions $f$ with $|f|_{\sup}<1$ and $\sO^*(1)\subset \sO^*$ is the subsheaf
of functions $f$ with $|1-f|_{\sup}<1$. Based on this isomorphism \cite[Satz 4]{Ge} reduces
Theorem~\ref{thm1} in case of characteristic zero to a vanishing result for the additive rigid cohomology group
$H^1(\Sp(A), \sO(1))$ which is established in \cite{B1}. As the articles \cite{B1} and \cite{B2}
are written in German and are not easy to read, we  give a simplified proof of their main
results in Section~\ref{sec.vanad} based on the cohomology theory of affinoid spaces \cite{vdP}. 

However in case ${\rm ch}(K)>0$ this approach using the exponential isomorphism does not
apply. Instead, in Section~\ref{sec.vanishing-mult} we
explain how to pass from a vanishing result for the additive  cohomology groups to a
vanishing result for the multiplicative cohomology groups in the absence of an exponential
isomorphism. Based on the latter vanishing the proof of Theorem~\ref{thm1} is given in
Section~\ref{sec.pic}.

In Section~\ref{sec.k0} we prove the following  stable version of Question~\ref{qu.main}.
Assume that $K$ is discretely valued, and hence its valuation ring is noetherian.
Let  $A^{\circ}$ denote the subring of power bounded elements in $A$.
By a regular model for a regular affinoid $K$-algebra $A$ we mean a proper morphism 
of schemes $\sX \to \Spec(A^{\circ})$ which is an isomorphism over $\Spec(A)$ and such that $\sX$ is regular.

\begin{thm}\label{thm2}
Let $K$ be discretely valued, 
and let $A/K$ be a regular affinoid algebra. Assume that $A$ admits a regular model; this is automatic if the residue field of $K$ has characteristic zero.
Then
\[
K_0(A) \to \prolim{t\mapsto \pi t}  K_0(A\langle t \rangle)
\]
is a pro-isomorphism.
\end{thm}

The proof of Theorem~\ref{thm2} uses ``pro-cdh-descent'' \cite{KST,Mo} for
the $K$-theory spectrum of schemes and resolution of singularities in the residue characteristic zero case;
so it is rather non-elementary. Of course, in the cases where Theorem~\ref{thm2} applies it comprises Theorem~\ref{thm1}, as there is a
surjective determinant map $\det: K_0 \to \Pic$.

\bigskip

{\it Acknowledgement.}\\
We would like to thank M.\ van der Put for helpful comments.

\subsection*{Notations}
We denote the supremum seminorm \cite[Sec.~3.1]{B} of a rigid analytic function $f$ on an affinoid space $X$ by $|f|_{\sup}$.
For a real number $r>0$ we denote by $\sO_{X}(r) \subseteq \sO_{X}$ the subsheaf of functions of supremum seminorm $<r$.   
We often omit the subscript $X$ if no confusion is possible. We write $\sO^{\circ} \subseteq \sO$ for the subsheaf of functions of supremum norm $\leq 1$. 

If $0<r<1$, functions of the from $1+f$ with $|f|_{\sup} <r$ are invertible, and we denote by $\sO^{*}(r) \subseteq \sO^{*}$ the subsheaf of invertible functions of this form.

We use similar notations $K(r), K^{\circ}, K^{*}(r)$ for corresponding elements of the field $K$ or complete valued extensions of $K$.

If $a$ is an analytic point of an affinoid space \cite[Sec.~2.1]{dJvdP}, we denote the completion of its residue field by $F_{a}$. 

For the closed polydisk $\Sp(K\langle t_{1},\dots, t_{d}\rangle)$ of radius 1 and dimension $d$ over $K$ we use the notation $\B^{d}_{K}$ or simply $\B^{d}$.

An affinoid algebra $A/K$ is called smooth if $A\otimes_K  K'$ is regular for all
finite field extensions $K\subset K'$. As a general reference concerning the terminology
of rigid spaces we refer to \cite{B}.

\section{Vanishing of additive cohomology (after Bartenwerfer)}\label{sec.vanad}

The aim of this section is to give new, more conceptual proofs of the main results of
\cite{B1} and~\cite{B2}. Our techniques are based on the cohomology theory for affinoid
spaces as developed by van der Put, see \cite{vdP} and \cite{dJvdP}.
Let $K$ be a field which is complete with respect to the non-archimedean absolute value
$|\cdot |:K\to\mathbb R$. We assume that  the absolute value $|\cdot |$ is not trivial.
All affinoid spaces we consider in this section are assumed to be integral.

Let $\mathcal M, \mathcal N$ be  sheaves of $\sO^\circ$-modules on the affinoid space $X=\Sp(A)$.
We say that $\mathcal M$ is weakly trivial if there exists $r\in (0,1)$ with
$\sO(r)\mathcal M=0$. Note that this just means that there exists $f\in K^\circ\setminus\{
0\}$ with $f \mathcal M=0$. The weakly trivial $\sO^{\circ}$-modules form a Serre subcategory 
of the abelian category of all sheaves of $\sO^{\circ}$-modules.
   We say that an  $\sO^\circ$-morphism $u:\mathcal M\to \mathcal N$ is
a weak isomorphism
if ${\rm coker}(u)$ and ${\rm ker}(u)$ are weakly
trivial. Note that the weak isomorphisms are exactly those morphisms which are invertible
up to multiplication by elements of
$K^\circ\setminus \{ 0\}$. We say that $\sM$ is weakly locally free ({wlf}) if there is a finite affinoid covering
$X=\cup_{i\in I} U_i$ and weak isomorphisms $(\sO^\circ_{U_i})^{n_i}\simeq \sM|_{U_i}$
for each $i\in I$.

Note that for $\sM$ wlf the $\sO_X$-module sheaf $\sM\otimes_{\sO^\circ_X} \sO_X$ is
coherent and locally free, i.e.~locally free of finite type.

\begin{lem}\label{lem.coker}
Let $\psi: \sM \to \sN$ be an $\sO^{\circ}$-morphism of wlf sheaves on $X=\Sp(A)$, and let $f\in A^{\circ}$. If    
\[
f \coker(\psi\otimes 1: \sM\otimes_{\sO^{\circ}} \sO \to \sN \otimes_{\sO^{\circ}} \sO) = 0,
\]
then there exists $r\in (0,1)$ such that $fK(r) \coker(\psi) =0$.
\end{lem}

\begin{proof}
By the definition of weak local freeness, we may assume without
 loss of generality that $\sM= (\sO^\circ)^m$ and $\sN= (\sO^\circ)^n$. Let $\sC$ be the
cokernel of $\psi$. By Tate's acyclicity theorem \cite[Cor.~4.3.11]{B} we get an exact sequence
\[
H^0(X,\sM\otimes_{\sO^\circ} \sO)  \to H^0(X,\sN\otimes_{\sO^\circ} \sO) \to H^0(X,\sC\otimes_{\sO^\circ} \sO),
\]
where the right hand $A$-module is $f$-torsion by assumption. Let $e_1,\ldots, e_n\in \sN(X)$ be
the canonical basis elements. So we deduce that $fe_1,\ldots , fe_n$ have preimages $l_1,\ldots,
l_n\in H^0(X,\sM\otimes_{\sO^\circ} \sO)=A^m$. Choose $r\in (0,1)$ such that
$K(r)l_1, \ldots, K(r)l_n\subset (A^\circ)^m$.
\end{proof}

\begin{prop}\label{prop.wlf}
Let $\sM$ be an $\sO^\circ$-module sheaf on $X=\Sp(A)$ such that $\sM\otimes_{\sO^\circ_X} \sO_X$ is
coherent and
locally free as $\sO_X$-module sheaf. Then the following are equivalent:
\begin{itemize}
\item[(i)] $\sM$ is wlf.
\item[(ii)] For each finite set of points $R\subset X$ there is
an injective $\sO^\circ$-linear morphism $\Psi:(\sO^{\circ})^n\to \sM$ and $f\in \sO^\circ(X)$ with
$f(x)\ne 0$ for all $x\in R$ such that  $f\, {\rm coker} (\Psi)=0$.
\item[(iii)] For each point $x\in X$ there is
an injective $\sO^\circ(X)$-linear morphism $\Psi_x:(\sO^{\circ})^n\to \sM$ and $f_x\in \sO^\circ(X)$ with
$f_x(x)\ne 0$ such that $f_x\, {\rm coker} (\Psi)=0$.
\end{itemize}
\end{prop}

\begin{proof}
Clearly, (ii) implies (iii). We first prove (iii) implies (i). Choose for each point $x\in
X$ a map $\Psi_x$ and $f_x$ as in (iii). There is a finite set of points $x_1,\ldots,
x_k\in X$ such that we get a Zariski covering 
\[
X=\bigcup_{i\in\{1,\ldots , k\}} \{x\in X\  | \ f_{x_i}(x)\ne 0 \}.
\]
By \cite[Lem.~5.1.8]{B} there exists $\epsilon \in \sqrt{|K^\times|}$ such that the $U_i=\{
  x\in X\ |\ |f_{x_i}(x)| \ge \epsilon \}$ cover $X$. Then the morphisms
$\Psi_{x_i}|_{U_i}$ are weak isomorphisms, so $\sM$ is wlf.

\smallskip

We now prove that (i) implies (ii). As  $\sM\otimes_{\sO^\circ_X} \sO_X$ is locally free, there
exists a finitely generated projective $A$-module $M$ with $M^\sim =  \sM\otimes_{\sO^\circ_X}
\sO_X$, \cite[Sec.~6.1]{B}. By $A_R$ we denote the semi-local ring which is the localization of $A$ at
the finitely many maximal ideals $R$. Choose a basis $b_1,\ldots , b_n$ of the free
$A_R$-module $M\otimes_A A_R$.
Without loss of
generality we can assume $b_1,\ldots , b_n$ are induced by elements of $ \sM(X)$. We claim
that  the latter elements give rise to  a
morphism $\Psi$ as in (ii). Indeed, by elementary algebra we find $f'\in A^\circ$ such that $f'(x)\not=0$ for all
$x\in R$ and such that 
\[
f'\, {\rm coker}( A^n\to M )=0.
\]
We conclude by Lemma~\ref{lem.coker}.
\end{proof}

\begin{prop}\label{prop.etale}
Let $\phi:X\to Y$ be a finite \'etale morphism of affinoid spaces over $K$ and let $\sM$ be a
wlf $\sO^\circ_X$-module. Then $\phi_* \sM$ is a wlf $\sO_Y^\circ$-module. 
\end{prop}

\begin{proof}
Let $X=\Sp(A)$ and $Y=\Sp(B)$.
The $\sO_Y$-module sheaf  $\phi_* (\sM) \otimes_{\sO_Y^\circ} \sO_Y= \phi_* (\sM
\otimes_{\sO_X^\circ} \sO_X)$ is coherent and locally free. For $y\in Y$ let $R$ be the
finite set $\phi^{-1}(y)$ and let $M\subset B$ be the maximal ideal corresponding to $y$. From Proposition~\ref{prop.wlf} we deduce that there is an
injective $\sO^\circ_X$-linear morphism 
\[
\Psi:(\sO^\circ_X)^n \to \sM
\]
whose cokernel is killed by some $f\in A^\circ$ which does not vanish on $R$. Then as
the induced homomorphism
$\phi^\sharp:B\to A$ is finite 
the prime ideals of $B$ containing the ideal $I= (\phi^\sharp)^{-1}( A f)$ are
exactly the preimages of the prime ideals in $A$ which contain $f$, see
\cite[Sec.~V.2.1]{BCA}. So we can find $g\in I\cap B^\circ$ which is not contained in $M$. Then the cokernel of the injective morphism
\[
\phi_*(\Psi):\phi_*( \sO^\circ_X)^n \to \phi_*(\sM).
\]
is $g$-torsion.
By  Proposition~\ref{prop.wlf} we see that  it suffices to show that $\phi_*(\sO_X^\circ)$ is
wlf.

Note that for $V\subset Y$ an affinoid subdomain $\sO_X^\circ(\phi^{-1}(V))$ is the
integral closure of $\sO_Y^\circ(V)$ in $A\otimes_{B} \sO_Y(V)=\sO_X(\phi^{-1}(V))$
\cite[Thm.~3.1.17]{B}. As
the field extension $Q(B)\to Q(A)$ is separable, it is not hard to bound this integral
closure as follows.
Let $b_1,\ldots , b_d\in \sO^\circ(X)$ induce a basis of the free $B_M$-module $A\otimes_B
B_M$. This basis induces an injective $\sO_Y^\circ$-linear morphism
\[
\Psi:(\sO_Y^\circ)^d \to \phi_*(\sO^\circ_X).
\]
Let $\delta$ be the discriminant of $b_1,\ldots , b_d$.
Then by \cite[Lem.~V.1.6.3]{BCA} the cokernel of $\Psi$ is $\delta$-torsion.

As the point $y\in Y$  was arbitrary we conclude from Proposition~\ref{prop.wlf} that  $\phi_*(\sO_X^\circ)$ is
wlf.
\end{proof}

In the proofs of Theorems~\ref{thm.van1} and \ref{thm.van2} below, we want to apply a base change theorem of van der Put (\cite[Thm.~2.7.4]{dJvdP}) and argue with stalks.
The latter work well if one restricts to overconvergent sheaves and analytic points, see \cite[Sec.~2]{dJvdP} for the definition and basic properties.
For a sheaf $\sM$ on $X$ we write $\sM^{\rm oc}$ for the associated overconvergent sheaf.
The sheaf  $\sM^{\rm oc}$ is given on an affinoid open subdomain $U\subset X$ by 
\[
\sM^{\rm oc}(U)= \operatorname{colim}_{U\subset  U'} \sM(U')
\]
where $U'$ runs through all wide neighborhoods of $U$ in $X$ (see \cite[Sec.~2.3]{dJvdP} for a definition). Note that there is a
canonical morphism $\sM^{\rm oc}\to \sM$. 

\begin{rmk}
Let $X = \Sp(A)$ be an affinoid rigid space over $K$, and let $X^{\mathrm{an}}$ be the Berkovich spectrum of $A$. The  analytic points of $X$ are in canonical bijection with the points of the topological space $X^{\mathrm{an}}$, and there is a morphism of topoi $(\sigma_{*}, \sigma^{*}) : X^{\sim} \to X^{\mathrm{an},\sim}$. The left adjoint $\sigma^{*}$ identifies $X^{\mathrm{an},\sim}$ with the full subcategory of $X^{\sim}$ consisting of overconvergent sheaves, and for any sheaf $\sM$ on $X$ the counit $\sigma^{*}\sigma_{*}\sM \to \sM$ is identified with the canonical map $\sM^{\rm oc}\to \sM$. The stalk of $\sigma_{*}\sM$ in a point of $X^{\mathrm{an}}$ is precisely the stalk of $\sM$ in the corresponding analytic point. Finally, for an overconvergent abelian sheaf $\sM$ on $X$ one has a natural isomorphism $H^{*}(X, \sM) \simeq H^{*}(X^{\mathrm{an}}, \sigma_{*}\sM)$ and similarly for higher direct images. Using this, van der Put's base change theorem for overconvergent sheaves can be deduced from the ordinary proper base change theorem in topology. See \cite{PutSch,Sch} for all this.
\end{rmk}

The following proposition is a simple
consequence of Tate's acyclicity theorem \cite[Cor.~4.3.11]{B}.

\begin{prop}\label{prop.fromtate}
Let $X=\Sp(A)$ be an affinoid space.
\begin{itemize}
\item[(i)]
For any finite affinoid covering $\mathcal U$ of $X$ the \v{C}ech cohomology groups $H^i(\sU,\sO^\circ)$ are weakly trivial (as $K^\circ$-modules)
 for all $i>0$.
\item [(ii)]
The canonical map 
\[
H^i(V,\sO_X(r)^{\rm oc}|_V)\to H^i(V,\sO_V(r))
\]
is surjective for every affinoid subdomain $V\subset X$, every  $r>0$ and integer $i > 0$.
\end{itemize}
\end{prop}

\begin{proof}
 (i):
Note that for each affinoid open subdomain $U $ of $X$ the 
 \v Cech complex $(C(\sU, \sO),d)$ consists of complete normed $K$-vector spaces and
the differential is continuous. To be concrete, we work with the supremum norm. The continuous morphism
\[
d^{i-1}:C^{i-1}(\sU,\sO)\to Z^i(\sU,\sO)
\]
is surjective by \cite[Cor.~4.3.11]{B}, so it is open according to
\cite[Thm.~I.3.3.1]{BTVS}. In other words there exists $r\in (0,1)$ such that
$Z^i(\sU,\sO(r))$ is contained in $d^{i-1}(C^{i-1}(\sU,\sO^\circ))$. This means that $H^i(\sU,\sO^\circ)$ is $K(r)$-torsion.

\medskip

 (ii): 
In order to show part (ii) of the proposition it suffices to show that for each finite
 covering $\sU=(U_l)_{l\in L}$ of $V$ by rational subdomains of $X$  the map
\begin{equation}\label{eq.oc}
H^i(\sU,\sO_X(r)^{\rm oc})\to H^i(\sU,\sO(r))
\end{equation}
is surjective. This is a consequence of
\begin{claim}\label{claim.oc}
\mbox{}
\begin{itemize}
\item[(i)] For $i>0$ the image of $d^{i-1}:C^{i-1}(\sU,\sO(r))\to Z^i(\sU,\sO(r))$ is  open.
\item[(ii)] The image of $Z^i(\sU, \sO_X(r)^{\rm oc}) \to Z^i(\sU, \sO(r)) $ is dense.
\end{itemize}
\end{claim}
Part (i) of the claim is a consequence of Proposition~\ref{prop.fromtate}(i). For part
(ii) of the claim note that
 for each rational subdomain 
\[
U= \{|g_1|\le |g_0|, \ldots ,
|g_r|\le |g_0|\}
\]
 of $X$ the image of $\sO_X^{\rm
  oc}(U)\to \sO(U)$ is dense. To see this observe that for $\epsilon >1$ and $\epsilon \in
|K^*|^{\mathbb Q}$ the set $U$ is a Weierstra\ss\ domain inside 
$\{|g_1|\le \epsilon |g_0|, \ldots ,
|g_r|\le \epsilon |g_0|\}$.

For $\xi \in Z^i(\sU, \sO(r))$ we find $\xi'\in C^{i-1}(\sU,\sO)$ with $d(\xi')=\xi$, using
again \cite[Cor.~4.3.11]{B}. Find a sequence $\xi'_j\in C^{i-1}(\sU,\sO_X^{\rm oc})$ such
that its image in $C^{i-1}(\sU,\sO)$ converges to $\xi'$. Then $d(\xi'_j)\in
Z^i(\sU,\sO^{\rm oc})$ is a sequence approximating $\xi$. By \cite[Lem.~2.3.1]{dJvdP} for
large $j$ we have $d(\xi'_j)\in
Z^i(\sU,\sO_X(r)^{\rm oc})$.
\end{proof}

\begin{thm}[Bartenwerfer/van der Put]\label{thm.van1}
We have 
\[
H^i(\mathbb B^d,\sO(r))=0
\] 
for all $r>0$ and integers $i>0$.
\end{thm}
This is proven by Bartenwerfer \cite[Theorem]{B2} and using different methods by van der Put \cite[Thm.~3.15]{vdP}. For the convenience of the reader, we sketch van der Put's proof.
\begin{proof}[Idea of proof (van der Put)]
Using Tate's acyclicity theorem the theorem is equivalent to the following two statements:
\begin{itemize}
\item for all $r>0$ and integers $i>0$ the cohomology group $$H^i(\mathbb
  B^d,\sO/\sO(r))=0,$$ 
\item $H^0(\mathbb B^d, \sO) \to H^0(\mathbb B^d, \sO/\sO(r))$ is surjective.
\end{itemize}
The sheaf $\sO/\sO(r)$ is overconvergent  by \cite[Lem.~1.5.2]{vdP}. 
So we can apply base change \cite[Thm.~2.7.4]{dJvdP} for the the linear fibrations $\phi:\mathbb B^d\to \mathbb B^{d-1}$.
Using the fact that for any fibre $\phi^{-1}(a)\cong \mathbb B^1_{F_a}$ over an analytic
point $a$ of $ \mathbb
B^{d-1}$ we have \[(\sO_{\mathbb B^d}/\sO_{\mathbb B^d}(r))|_{\phi^{-1}(a)} \cong
  \sO_{\mathbb B^1_{F_a}}/\sO_{\mathbb B^1_{F_a}}(r) , \]
compare Lemma~\ref{lem.stalks}, we reduce the theorem to the case $d=1$. In fact, by what
is sayed and using
the one-dimensional case of the theorem we get that 
\begin{align*}
\phi_*(\sO_{\mathbb B^d}/\sO_{\mathbb B^d}(r))  &=  \bigoplus_{\mathbb N}  \sO_{\mathbb
                                                  B^{d-1}}/\sO_{\mathbb B^{d-1}}(r) ,\\
R^j \phi_*(\sO_{\mathbb B^d}/\sO_{\mathbb B^d}(r))  &=  0 \quad (j>0)
\end{align*}
and we conclude by the Leray spectral sequence and  by induction on $d$.
 
In the one-dimensional case the theorem follows from an explicit
computation based on the Mittag--Leffler decomposition. 
\end{proof}

\begin{cor}\label{cor.van}
The cohomology group
\[
H^i(\mathbb B^d,\sO^\circ)
\] 
is $K(1)$-torsion for all integers $i>0$.
\end{cor}
Indeed, for any $\alpha\in K(1)$ the multiplication by $\alpha$ on $H^{i}(\mathbb B^d,\sO^\circ)$ factors through $H^{i}(\mathbb B^d,\sO(1))$ which vanishes by Theorem~\ref{thm.van1}.
\begin{rmk}
In fact, in \cite[Thm.]{B3} Bartenwerfer shows that $H^{i}(\mathbb B^{d}, \sO^{\circ}) = 0$ for every $i >0$. 
\end{rmk}

\begin{lem}\label{lem.retri}
Let $X=\Sp(A)$ be an affinoid space such that the cohomology group $H^i(X,\sO^\circ)$ is weakly
trivial for some $i>0$. Then for any wlf $\sO^\circ$-module $\sM$ the cohomology group $H^i(X,\sM)$ is weakly
trivial.
\end{lem}

\begin{proof}
Below we are going  to construct for every point $x\in X$ a function $f_x\in A^\circ$ with
$f_x(x)\ne 0$ and with $f_x\, H^i(X,\sM)=0$. As the $f_x$ generate the unit ideal in $A$,
there exist finitely many points $x_1,\ldots , x_r\in X$ and $c_1,\ldots , c_r\in A^\circ$ with 
\[
c_1 f_{x_1} + \cdots + c_r f_{x_r}=:c\in K^\circ \setminus \{ 0\}.
\]
Then $c\, H^i(X,\sM)=0$.

In order to construct such $f_x$ for given $x\in X$ we use Proposition~\ref{prop.wlf} in
order to find an injective $\sO_{X}^\circ$-linear morphism $\Psi:(\sO^{\circ})^n\to \sM$ and $f'\in \sO^\circ(X)$ with
$f'(x)\ne 0$ and such that $f'\, {\rm coker} (\Psi)=0$.  From the long exact cohomology
sequence corresponding to the short exact sequence
\[
0\to (\sO^{\circ})^n \xrightarrow{\Psi} \sM \to {\rm coker}(\Psi)\to 0
\]
it follows that we can take any nonzero $f_x\in K(r) f'$, where $r\in (0,1)$ is chosen
such that $K(r) \, H^i(X,\sO^\circ)=0$.
\end{proof}

\begin{thm}\label{thm.van2}
For $X/K$ a smooth affinoid space  and for $\sM$ a wlf $\sO^\circ_X$-module the  cohomology groups
$H^i(X,\sM)$ are weakly trivial (as $K^\circ$-modules) for all $i>0$.
\end{thm}

\begin{proof}
By Lemma~\ref{lem.retri} we can assume without loss of generality that $\sM=\sO^\circ$.
We use induction on $i>0$. The base case $i=1$ is handled in the same way as the
induction step, so let us assume $i>1$ and that we already know weak triviality of
$H^j(U,\sO^\circ)$  for all $0<j<i$ and smooth affinoid spaces $U/K$.

Since $X/K$ is smooth,  \cite[Satz 1.12]{K} implies that there exists a finite affinoid covering $\sU=(U_l)_{l\in L}$ and finite
\'etale morphisms $\phi_l:U_l\to \mathbb B^d$.
From the \v{C}ech spectral sequence \[E_{2}^{pq}=\checkH^{p}(\sU, \underline{H}^{q}(\sO^\circ)) \Rightarrow H^{p+q}(X, \sO^\circ)\] we see that $H^{i}(X,\sO^\circ)$ has a filtration whose associated graded piece $\operatorname{gr}^{p}$ is a subquotient of $\checkH^{p}(\sU, \underline{H}^{i-p}(\sO^\circ))$. 
By Proposition~\ref{prop.fromtate}(i), $\operatorname{gr}^{i}$ is weakly trivial. By our
induction assumption,  $\underline{H}^{i-p}(\sO^\circ)(U)$ is weakly trivial for $0<p<i$
and for $U$ an intersection of opens in $\sU$, hence $\operatorname{gr}^{i-p}$ is weakly trivial for these $p$. It thus suffices to show that $\operatorname{gr}^{0}$ is weakly trivial or that 
$H^i(U_l,\sO^\circ_{U_l}) $ is weakly trivial for all $l\in L$.

\smallskip

So in order to show Theorem~\ref{thm.van2}  we can assume  without loss of generality that $\sM=\sO^\circ_X$ and that there exists a
finite \'etale morphism $\phi:X\to
\mathbb B^d$.

 For all $j > 0$ we get morphisms  
\begin{equation}\label{eq.vanhi}
R^j\phi_{*}( \sO_{X}^\circ) \simeq R^j\phi_{*}( \sO_{X}(1)) \leftarrow  R^j\phi_{*}( \sO_{X}(1)^{\rm oc}).
\end{equation}
with a weak isomorphism on the left and a surjective morphism on the right. The
surjectivity follows from Proposition~\ref{prop.fromtate}(ii).
By base change  \cite[Thm.~2.7.4]{dJvdP} the stalk $R^j\phi_{*}( \sO_{X}(1)^{\rm oc})_{a} \simeq H^{j}(X_{a}, \sO_{X}(1)^{\mathrm{oc}}|_{X_{a}})$ 
vanishes for every analytic point $a$ of $\mathbb{B}^{d}$. Since $R^j\phi_{*}( \sO_{X}(1)^{\rm oc})$ is overconvergent \cite[Lem.~2.3.2]{dJvdP}, it follows that $R^j\phi_{*}( \sO_{X}(1)^{\rm oc}) = 0$ and hence  that $R^j\phi_{*}( \sO_{X}^\circ)$ is weakly trivial.

Combining this observation with the Leray spectral sequence we see
that it suffices to show that
$H^i(\mathbb B^d,\phi_* ( \sO_{X}^\circ))$ is weakly trivial for $i>0$. From
Proposition~\ref{prop.etale} we deduce that  $\phi_* (\sO_{X}^\circ)$
is wlf as an $\sO_{\mathbb B^d}^\circ$-module, so we conclude by using Theorem~\ref{thm.van1}
and Lemma~\ref{lem.retri}.
\end{proof}

The following corollary, which we will apply in the next sections, was first shown in
\cite{B1} and \cite[Folgerung~3]{B2}.

\begin{cor}[Bartenwerfer]\label{cor.bartenwerfer}
For $X/K$ smooth affinoid there exists $s\in (0,1)$ such that the map 
\begin{equation}\label{eq.bart}
H^i(X,\sO(sr))\to H^i(X,\sO(r))   
\end{equation}
vanishes for all $r>0$ and integers $i>0$.
\end{cor}

\begin{proof}
Choose $\pi\in K(1)\setminus \{ 0\}$ and write $s'=|\pi|$. By Theorem~\ref{thm.van2} we
can assume without loss of generality that $\pi\, H^i(X,\sO(1))=0$ for  $i>0$. Now we
claim $s=s'^2$ satisfies the requested property of the corollary. Indeed, for $r>0$ set
$r'=\max\{|\pi|^n\ |\ n\in \mathbb Z, |\pi|^n\le r \}$. Then we get a commutative square
\[
\xymatrix{
H^i(X,\sO(s'r')) \ar[r]\ar[d]^{\wr}  & H^i(X,\sO(r')) \ar[d]^{\wr} \\
H^i(X,\sO(1)) \ar[r]^{=0} &  H^i(X,\sO(1))
}
\]
where the lower horizontal map is multiplication by $\pi$ and the vertical maps are
induced by the isomorphisms $\sO(s'r') \cong \sO(1)$ and $\sO(r') \cong \sO(1)$ given by
multiplying with the appropriate powers of $\pi$. The morphism \eqref{eq.bart} is the
composition of
\[
H^i(X,\sO(sr)) \to H^i(X,\sO(s'r'))\xrightarrow{=0} H^i(X,\sO(r')) \to H^i(X,\sO(r)). 
\]
\end{proof}

\section{Vanishing of multiplicative cohomology}\label{sec.vanishing-mult}

Given $r'<r$ we write $\sO(r,r') := \sO(r)/\sO(r')$ and, if $r'<r\leq1$, $\sO^{*}(r,r') := \sO^{*}(r)/\sO^{*}(r')$.
\begin{lem}\label{lem.mult-add}
For $r' < r \leq 1$ we have isomorphisms of sheaves of sets $\sO(r) \xrightarrow{\sim} \sO^{*}(r)$ and $\sO(r,r') \xrightarrow{\sim} \sO^{*}(r,r')$ given by $f\mapsto 1+f$. If $r'\geq r^{2}$, the latter isomorphism is an isomorphism of abelian sheaves.
\end{lem}
\begin{proof}
Most of the claims are easy. 
To see that $f\mapsto 1+f$ induces a map on the quotient sheaves $\sO(r,r') \to \sO^{*}(r,r')$ note that 
 if $f,g$ are functions of supremum seminorm $<1$, then $|f-g|_{\sup} < r'$ if and only if $|(1+f)(1+g)^{-1}-1|_{\sup} < r'$. Indeed, this follows from
the computation $|f-g|_{\sup} = |(1+f) - (1+g)|_{\sup} = |((1+f)(1+g)^{-1} - 1)(1+g)|_{\sup} = |(1+f)(1+g)^{-1} - 1|_{\sup}$, where we used that $|1+g|_{\sup} = |(1+g)^{-1}|_{\sup} = 1$. 
\end{proof}

Given an affinoid space $X$, we consider the following condition on the real number $0<s\leq 1$: 
\begin{align}\label{condition-on-s}\begin{split}
&\text{The map } H^i(X,\sO(sr))\to H^i(X,\sO(r))  \\ 
&\text{vanishes for all $r>0$ and integers $i>0$.}
\end{split}
\end{align}

\begin{prop}\label{prop.mult-vanishing}
Let  $X/K$ be  smooth affinoid. Assume that  $s$ satisfies \eqref{condition-on-s}.
Then  the map
\[
H^1(X, \sO^{*}(sr)) \to H^{1}(X, \sO^{*}(r))
\]
vanishes for every  $r \in (0,s)$.
\end{prop}
\begin{proof}

We first prove:

\begin{lem}
Assume that $s$ satisfies \eqref{condition-on-s} for the affinoid space $X$.
For any integer $i>0$, $r\in (0,s)$, and $\xi \in H^i(X,\sO^{*}(sr))$ there exists a decreasing  zero sequence $(r_{n})$ in $(0,s)$ with $r_{0}=r$  and a compatible system
\[
(\xi'_{n}) \in \lim_{n} H^{i}(X, \sO^{*}(r_{n}))
\]
such that $\xi'_{0} \in H^{i}(X,\sO^{*}(r))$ is equal to the image of $\xi$ under $H^{i}(X,\sO^{*}(sr)) \to H^{i}(X,\sO^{*}(r))$.
\end{lem}

\begin{proof}
Put $r_{0}=r$ and inductively $r_{n+1}= r_{n}^{2}/s$. Explicitly, $r_{n} = (r/s)^{2^{n}}s$. Since $r<s$,  the $r_{n}$ form a decreasing zero sequence.

Put $\xi_{0}=\xi$. We will inductively construct elements $\xi_{n} \in H^{i}(X, \sO^{*}(sr_{n}))$ such that the images of $\xi_{n}$ and $\xi_{n+1}$ in $H^{i}(X, \sO^{*}(r_{n}))$ coincide. Denote this common image by $\xi'_{n}$. Then $(\xi'_{n})_{n\geq 0}$ is the desired compatible system.

Assume that we have already constructed $\xi_{n}$.
From the commutative diagram with exact rows
\[
\xymatrix{
H^{i}(X, \sO( sr_{n} )  ) \ar[r]\ar@{=}[d] & H^{i}(X, \sO(sr_{n},s^{2}r_{n+1}  ) ) \ar[r]\ar[d] & H^{i+1}(X, \sO(s^{2}r_{n+1}  ) ) \ar[d]^{=0 \text{ by~\eqref{condition-on-s}}}    \\
H^{i}(X, \sO( sr_{n} ) ) \ar[r]\ar[d]^{=0 \text{ by~\eqref{condition-on-s}}} & H^{i}(X, \sO( sr_{n},sr_{n+1} )  ) \ar[r]\ar[d] & H^{i+1}(X, \sO(sr_{n+1}  ) ) \ar@{=}[d]    \\
H^{i}(X,  \sO( r_{n} )) \ar[r]    & H^{i}(X,  \sO( r_{n},sr_{n+1} )) \ar[r]   & H^{i+1}(X, \sO(sr_{n+1}  )  )   \\
}
\]
we see that $H^{i}(X, \sO(sr_{n},s^{2}r_{n+1})) \to H^{i}(X, \sO(r_{n},sr_{n+1}))$ vanishes for $i>0$.
Since $sr_{n+1} \geq r_{n}^{2}$ and $s^{2}r_{n+1} = sr_{n}^{2} \geq (sr_{n})^{2}$, we may apply Lemma~\ref{lem.mult-add} to deduce that also $H^{i}(X, \sO^{*}(sr_{n},s^{2}r_{n+1})) \to H^{i}(X, \sO^{*}(r_{n},sr_{n+1}))$ vanishes.
From  the commutative diagram with exact rows
\[
\xymatrix{
& H^{i}(X, \sO^{*}( sr_{n} )) \ar[d] \ar[r] & H^{i}(X, \sO^{*}(sr_{n}, s^{2}r_{n+1})) \ar[d]^{=0} \\
H^{i}(X, \sO^{*}(sr_{n+1})) \ar[r] & H^{i}(X, \sO^{*}(r_{n})) \ar[r] & H^{i}(X, \sO^{*}(r_{n},sr_{n+1}))
}
\]
we deduce the existence of the desired element $\xi_{n+1} \in H^{i}(X, \sO^{*}(sr_{n+1}))$ such that the images of $\xi_{n}$ and $\xi_{n+1}$ in $H^{i}(X,\sO^{*}(r_{n}))$ coincide. 
\end{proof}

\begin{lem}
Let $X/K$ be smooth affinoid, and
let $(\xi_{n}) \in \lim_{n} H^{1}(X, \sO^{*}(r_{n}))$ be a compatible system where the $r_{n}$ form a  decreasing zero sequence in $(0,1)$. 
Then there exists a finite affinoid covering  $\sU$ of $X$ such that $(\xi_{n})$ lies in the image of $\lim_{n} \checkH^{1}(\sU, \sO^{*}(r_{n}))$.
\end{lem}
\begin{proof}
Let $\sU$ be a finite affinoid covering of $X$ such that $\xi_{0}$ lies in the image of $\checkH^{1}(\sU, \sO^{*}(r_{0}))$.
We claim that then $\xi_{n}$ lies in the image of $\checkH^{1}(\sU, \sO^{*}(r_{n}))$ for all $n$.
Recall that  for any abelian sheaf $\sF$ the map $\checkH^{1}(\sU, \sF) \to H^{1}(X,\sF)$ is injective, and an element $\xi \in H^{1}(X,\sF)$ belongs to the image of this map if and only if $\xi|_{U}=0$ in $H^{1}(U, \sF|_{U})$  for every $U\in \sU$.

Fix $U \in \sU$. We want to show that $\xi_{n}|_{U} = 0$ in $H^{1}(U, \sO^{*}(r_{n}))$.
By Corollary~\ref{cor.bartenwerfer} there exists $m\geq n$  such that $H^{1}(U, \sO(r_{m})) \to H^{1}(U, \sO(r_{n}))$ vanishes. 
Under the sequence of maps 
\[
H^{1}(U, \sO^{*}(r_{m})) \to H^{1}(U, \sO^{*}(r_{n})) \to H^{1}(U, \sO^{*}(r_{0}))
\]
we have $\xi_{m}|_{U} \mapsto \xi_{n}|_{U} \mapsto 0$. Hence the element $\xi_{m}|_{U}$ lifts to an element $\eta_{m}$ in $H^{0}(U, \sO^{*}(r_{0}, r_{m}))$. We claim that the image of $\eta_{m}$ in $H^{0}(U, \sO^{*}(r_{0}, r_{n}))$ has a preimage in $H^{0}(U, \sO^{*}(r_{0}))$. In view of the commutative diagram with exact rows
\[
\xymatrix{
H^{0}(U, \sO^{*}(r_{0})) \ar[r] & H^{0}(U, \sO^{*}(r_{0}, r_{n})) \ar[r] & H^{1}(U, \sO^{*}(r_{n})) \\
H^{0}(U, \sO^{*}(r_{0})) \ar[r]\ar@{=}[u] & H^{0}(U, \sO^{*}(r_{0},r_{m})) \ar[r]\ar[u] & H^{1}(U, \sO^{*}(r_{m})) \ar[u]
}
\]
this will imply that $\xi_{n}|_{U}=0$.

To prove the claim, note that  Lemma~\ref{lem.mult-add} gives bijections $H^{0}(U, \sO^{*}(r_{0})) \cong H^{0}(U, \sO(r_{0}))$ and $H^{0}(U, \sO^{*}(r_{0}, r_{n})) \cong H^{0}(U, \sO(r_{0}, r_{n}))$ and similarly for $r_{n}$ replaced by $r_{m}$. On the other hand, by the choice of $m$, the map $H^{1}(U, \sO(r_{m})) \to H^{1}(U, \sO(r_{n}))$ vanishes. This implies the existence of the desired lift in view of the commutative diagram with exact rows
\[
\xymatrix{
H^{0}(U, \sO(r_{0})) \ar[r] & H^{0}(U, \sO(r_{0}, r_{n})) \ar[r] & H^{1}(U, \sO(r_{n})) \\
H^{0}(U, \sO(r_{0})) \ar[r]\ar@{=}[u] & H^{0}(U, \sO(r_{0},r_{m})) \ar[r]\ar[u] & H^{1}(U, \sO(r_{m})). \ar[u]_{=0} 
}
\]  
\end{proof}

We can now finish the proof of Proposition~\ref{prop.mult-vanishing}.
Using the two preceding lemmas, it suffices to show that $\lim_{n} \checkH^{1}(\sU, \sO^{*}(r_{n}))$ vanishes for every 
decreasing zero sequence $(r_{n})$. Consider an element $(\xi_{n})_{n}$  in this inverse limit, and choose representing \v{C}ech 1-cocycles $\zeta_{n} \in Z^{1}(\sU, \sO^{*}(r_{n}))$. Then there exist 0-cochains $\eta_{n} \in C^{0}(\sU, \sO^{*}(r_{n}))$ such that $\zeta_{n} = \zeta_{n+1}\cdot \partial \eta_{n}$.
Since $(r_{n})$ is a zero sequence, the product $\prod_{k=0}^{\infty}\eta_{n+k}$ converges in $C^{0}(\sU, \sO^{*}(r_{n}))$, and we get $\zeta_{n} = \partial(\prod_{k=0}^{\infty}\eta_{n+k})$, i.e. $\xi_{n}=0$.
\end{proof}

\begin{cor}\label{cor.multiplicative-vanishing-disks}
For every $r\in (0,1)$ we have $H^{1}(\B^{d}, \sO^{*}(r)) = 0$.
\end{cor}
\begin{proof}
By Theorem~\ref{thm.van1}, $s=1$ satisfies condition \eqref{condition-on-s} for $X=\B^{d}$.
Hence by 
Proposition~\ref{prop.mult-vanishing}, the identity map on $H^{1}(\B^{d}, \sO^{*}(r))$ vanishes.
\end{proof}

\begin{cor}\label{cor.mult-vanishing}
Let $X/K$ be a smooth affinoid space. Then there exists $0<r\leq 1$ such that
\[
H^{1}(X, \sO^{*}) \to H^{1}(X, \sO^{*}/\sO^{*}(r'))
\]
is injective for every $r'\in (0,r)$.
\end{cor}
\begin{proof}
By Corollary~\ref{cor.bartenwerfer} there exists $0< s \leq 1$ satisfying \eqref{condition-on-s}. By Proposition~\ref{prop.mult-vanishing} we can take $r=s^{2}$.
\end{proof}

\section{Homotopy invariance of \texorpdfstring{$\Pic$}{Pic}}
\label{sec.pic}

In this section we prove Theorem~\ref{thm1}.
Given $0 < r \leq 1$, we set $\sO^{*}(\infty,r) = \sO^{*}/\sO^{*}(r)$. 
Let $X = \Sp(A)$ be an affinoid space, and let $p: X\times \B^{1} \to X$ be the projection, $\sigma: X \to X \times \B^{1}$ the zero section. 

\begin{lem}\label{lem.stalks}
For any fibre $p^{-1}(a)\cong \B^{1}_{F_{a}}$ over an analytic point $a$ of $X$ we have
\[
\sO^{*}_{X\times\B^{1}}(\infty,r)|_{p^{-1}(a)} \cong \sO^{*}_{\B^{1}_{F_{a}}}(\infty,r).
\]
\end{lem}
\begin{proof}
This follows easily from \cite[Lemmas~2.7.1, 2.7.2]{dJvdP}.
\end{proof}

\begin{lem}\label{lem.vanishing-of-R1}
We have $R^{1}p_{*}\sO^{*}_{X\times\B^{1}}(\infty,r) =0$.
\end{lem}

\begin{proof}
The sheaf $\sO^{*}_{X\times\B^{1}}(\infty,r)$ and hence its higher direct images are overconvergent (see~\cite[1.5.3]{vdP}, \cite[Lem.~2.3.2]{dJvdP}). Hence it suffices to prove that for any analytic point $a$ of $X$ the stalk $R^{1}p_{*}\sO^{*}_{X\times\B^{1}}(\infty,r)_{a}$ vanishes. By base change \cite[Thm.~2.7.4]{dJvdP} and Lemma~\ref{lem.stalks}, we have
\[
R^{1}p_{*}\sO^{*}_{X\times\B^{1}}(\infty,r)_{a} \cong H^{1}(\B^{1}_{F_{a}}, \sO^{*}_{\B^{1}_{F_{a}}}(\infty,r)).
\]
In the exact sequence
\[
H^{1}(\B^{1}_{F_{a}}, \sO^{*}_{\B^{1}_{F_{a}}}) \to H^{1}(\B^{1}_{F_{a}}, \sO^{*}_{\B^{1}_{F_{a}}}(\infty,r)) \to H^{2}(\B^{1}_{F_{a}}, \sO^{*}_{\B^{1}_{F_{a}}}(r))
\]
the group on the left vanishes because the Tate algebra is a UFD, the group on the right vanishes by dimension reasons.
\end{proof}

Fix $\pi\in K\setminus\{0\}$ with $|\pi|<1$.
Let $t$ denote the coordinate on $\B^{1}$. Then $t \mapsto \pi t$ induces a map 
$p_{*}\sO^{*}_{X\times\B^{1}}(\infty,r) \to p_{*}\sO^{*}_{X\times\B^{1}}(\infty,r)$.
\begin{lem}\label{lem.pro-iso-direct-image}
We have an isomorphism of pro-abelian sheaves
\[
\prolim{t \mapsto \pi t} p_{*}\sO^{*}_{X\times\B^{1}}(\infty,r)  \cong \sO^{*}_{X}(\infty,r)
\]
\end{lem}
\begin{proof}
Obviously, $\sO^{*}_{X}(\infty,r) \xrightarrow{p^{*}} p_{*}\sO^{*}_{X\times\B^{1}}(\infty,r) \xrightarrow{\sigma^{*}} \sO^{*}_{X}(\infty,r)$ is the identity. 
Choose $n$ big enough such that $|\pi^{n}| \leq r$.
We claim that the  map
\[
p_{*}\sO^{*}_{X\times\B^{1}}(\infty,r) \to p_{*}\sO^{*}_{X\times\B^{1}}(\infty,r)
\]
induced by $t\mapsto \pi^{n}t$
factors through $\sO^{*}_{X}(\infty,r) \xrightarrow{p^{*}} p_{*}\sO^{*}_{X\times\B^{1}}(\infty,r)$.
By overconvergence again it is enough to check this on the stalk at any analytic point $a$ of $X$
(consider the image of the composition of the first map with the projection to $\coker(p^{*})$).
By base change and Lemma~\ref{lem.stalks} we have $p_{*}\sO^{*}_{X\times\B^{1}}(\infty,r)_{a} \cong H^{0}( \B^{1}_{F_{a}}, \sO^{*}_{\B^{1}_{F_{a}}}(\infty,r) )$.
By Corollary~\ref{cor.multiplicative-vanishing-disks} the natural map 
$H^{0}(\B^{1}_{F_{a}}, \sO^{*}) \to H^{0}( \B^{1}_{F_{a}}, \sO^{*}_{\B^{1}_{F_{a}}}(\infty,r) )$
is surjective.
Any element of $H^{0}(\B^{1}_{F_{a}}, \sO^{*})$ is of the form $u\cdot f(t)$ with $u \in F_{a}^{*}, f(0)=1$, and $|f(t) - 1|_{\sup}<1$ (see~\cite[Cor.~2.2.4]{B}). But then $|f(\pi^{n}t)-1|_{\sup}<|\pi^{n}|\leq r$. This implies that the map
\[H^{0}( \B^{1}_{F_{a}}, \sO^{*}_{\B^{1}_{F_{a}}}(\infty,r) ) \to H^{0}( \B^{1}_{F_{a}}, \sO^{*}_{\B^{1}_{F_{a}}}(\infty,r) )\]
induced by $t\mapsto \pi^{n}t$ factors through $F_{a}^{*}/F_{a}^{*}(r) \hookrightarrow H^{0}( \B^{1}_{F_{a}}, \sO^{*}_{\B^{1}_{F_{a}}}(\infty,r) )$, concluding the proof.
\end{proof}

\begin{proof}[Proof of Theorem~\ref{thm1}]
Note that $\Pic(A) \cong H^{1}(X,\sO^{*})$. 
Since $X=\Sp(A)$ is assumed to be smooth, Corollary~\ref{cor.mult-vanishing} implies that  there exists $r\in (0,1)$ such that the map $H^{1}(X\times \B^{1}, \sO^{*}) \to H^{1}(X\times \B^{1}, \sO^{*}(\infty,r))$ is injective.
It thus suffices to show that
\[
\sigma^{*}: \prolim{t\mapsto \pi t} H^{1}(X\times \B^{1}, \sO^{*}_{X\times\B^{1}}(\infty,r)) \to H^{1}(X, \sO^{*}_{X}(\infty,r))
\]
is a pro-isomorphism. 

Using the Leray spectral sequence, Lemma~\ref{lem.vanishing-of-R1} yields an isomorphism
\[
 H^{1}(X\times \B^{1}, \sO^{*}_{X\times\B^{1}}(\infty,r)) \cong H^{1}(X, p_{*}\sO^{*}_{X\times\B^{1}}(\infty,r)).
\]
We combine this with the pro-isomorphism 
\[
\prolim{t\mapsto \pi t} H^{1}(X, p_{*}\sO^{*}_{X\times\B^{1}}(\infty,r)) \cong  H^{1}(X, \sO^{*}_{X}(\infty,r))
\]
implied by Lemma~\ref{lem.pro-iso-direct-image} to finish the proof.
\end{proof}

\begin{proof}[Proof of Corollary~\ref{cor.A1invariance}]
Write $X$ for $\Sp(A)$, $U_{n}$ for the closed disk of radius $|\pi^{-n}|$, and $\mathbb{A}^{1,\mathrm{an}}$ for the analytic affine line over $K$. Then $X\times U_{n}$, $n=0,1,\dots$, is an admissible covering of $X\times \mathbb{A}^{1,\mathrm{an}}$. Note that the pro-systems $\prolim{n} \Pic(X\times U_{n})$ and $\prolim{t\mapsto \pi t}\Pic(A\langle t\rangle)$ are naturally isomorphic.
Taking the limit of the isomorphism of pro-abelian groups in Theorem~\ref{thm1} then gives the isomorphism
\[
\Pic(X) \cong \lim_{n} \Pic(X\times U_{n}).
\]
Hence it suffices to show that the natural map $\Pic(X\times \mathbb{A}^{1,\mathrm{an}}) \to \lim_{n} \Pic(X\times U_{n})$
is an isomorphism. The cohomological description of Picard groups yields a short exact sequence
\[
0 \to {\lim_{n}}^{1}\ \sO^{*}(X\times U_{n}) \to \Pic(X\times \mathbb{A}^{1,\mathrm{an}}) \to \lim_{n} \Pic(X\times U_{n}) \to 0.
\]
We have a natural decomposition $\sO^{*}(X\times U_{n}) \cong \sO^{*}(X) \oplus \sO^{*}_{0}(X\times U_{n})$ where $\sO^{*}_{0}(X\times U_{n})$ consists of those units that restrict to 1 on $X\subset X\times U_{n}$.
Clearly, $\lim^{1}_{n} \sO^{*}(X)=0$ and it remains to prove that $\lim_{n}^{1} \sO^{*}_{0}(X\times U_{n})$ vanishes. Note that given $f\in \sO^{*}_{0}(X\times U_{n+m})$, its restriction to $X\times U_{n}$ satisfies $|f|_{X\times U_{n}} -1|_{\sup} < |\pi^{m}|$.
Hence, given any sequence $(g_{n})_{n=0}^{\infty}$ with $g_{n}\in \sO^{*}_{0}(X\times U_{n})$, the product
\[
f_{n} := \prod_{k=n}^{\infty} g_{k}|_{X\times U_{n}} \in \sO^{*}_{0}(X\times U_{n})
\]  
converges. By construction we have $g_{n} = f_{n}\cdot (f_{n+1}|_{X\times U_{n}})^{-1}$ for every $n\geq 0$. This shows the desired vanishing of the $\lim^{1}$-term.
\end{proof}

\section{\texorpdfstring{$K_0$}{K0}-invariance}\label{sec.k0}

In this section we  assume that $K$ is a  complete  discretely valued field. 
Then for an affinoid algebra $A/K$ the ring of power bounded elements $A^\circ$ is noetherian, excellent, and of finite
Krull dimension, for excellence see~\cite[Sec.~I.9]{Il}.
Let
$\pi\in K^\circ$ be a prime element.

Let $\mathcal X\to \Spec(A^\circ)$ be a proper morphism of schemes which is an isomorphism over $\Spec(A)$.
For an integer
$n>0$ set $\sX_n=\sX \otimes_{K^\circ}K^\circ/(\pi^n)$.

\begin{prop}\label{prop.k0inj}
There exists $n>0$ such that 
\[
K_0(\sX) \to K_0( \sX_n)
\]
is injective.
\end{prop}

\begin{proof}
Let $K(\sX,\sX_n)$ be the homotopy fibre of the map $K(\sX)\to K(\sX_n)$ between
non-connective $K$-theory spectra \cite[Sec.~IV.10]{W} and let $K_i(\sX,\sX_n)$ be its
homotopy groups.
By ``pro-cdh-descent'' \cite[Thm.~A]{KST} the natural map 
\[
\prolim{n} K_0(A^\circ,A^\circ/(\pi^n))  \to \prolim{n} K_0(\sX,\sX_n)
\]
is a pro-isomorphism. For each $n$ we have an exact sequence
\[
K_1(A^\circ) \to K_1(A^\circ/(\pi^n)) \to K_0(A^\circ,A^\circ/(\pi^n)) \to K_0(A^\circ) \xrightarrow{\sim} K_0(A^\circ/(\pi^n))
\]
where the left map is surjective  \cite[Rmk.~III.1.2.3]{W} and the right map is an isomorphism  \cite[Lem.~II.2.2]{W}, since $A^{\circ}$ is $\pi$-adically complete. So $K_0(\sX,\sX_n)$ vanishes as a pro-system in $n$. By the exact  sequence
\[
 K_0(\sX,\sX_n) \to K_0(\sX) \to K_0(\sX_n)
\]
this finishes the proof of the proposition.
\end{proof}

\begin{lem}\label{lem.locseq}
If $\sX$ is a regular scheme we obtain a natural exact sequence
\[
G_0(\sX_1)\to K_0(\sX) \to K_0(A) \to 0,
\]
where $G_0$ is the Grothendieck group of coherent sheaves.
\end{lem}

\begin{proof}[Proof of Theorem \ref{thm2}]
In case the residue field of $K$ has characteristic zero, 
 $A^\circ$ contains $\mathbb Q$ and is excellent. Hence there exists a blow-up $\sX\to A^\circ$, whose center is (set theoretically) contained  in the closed fibre $\Spec(A^{\circ}/\pi)$,  such
that $\sX$ is a regular scheme \cite[Thm.~1.1]{Te}. So we can now assume in the general case that $\sX \to \Spec(A^{\circ})$ is a regular model of $A$ in the sense of the introduction.   Let $A^\circ\langle  t\rangle \subset
A^\circ\llbracket  t \rrbracket$ be
those formal power series for which the coefficients converge to zero. Note that
$A^\circ\to A^\circ\langle  t\rangle$ is a regular ring homomorphism, so $\sX'=\sX
\otimes_{A^\circ} A^\circ\langle  t\rangle $ is a regular scheme with generic fibre $\Spec
(A\langle t\rangle)$. Set $\sX'_n=\sX'
\otimes_{K^\circ}K^\circ/(\pi^n)$. 

Applying Lemma~\ref{lem.locseq} to $\sX$ and $\sX'$  we get a commutative diagram with
exact rows
\[
\xymatrix{
G_0(\sX_1) \ar[r]  &  K_0(\sX) \ar[r]  & K_0(A) \ar[r] & 0\\
G_0(\sX'_1) \ar[r] \ar[u]_{\wr}^{\sigma^*}  &  K_0(\sX') \ar[r] \ar[u]^{\sigma^*}  & K_0(A\langle  t\rangle)\ar[u]^{\sigma^*}  \ar[r] & 0
}
\]
where $\sigma$ is the zero-section induced by $t\mapsto 0$. The left vertical arrow is an
isomorphism by homotopy invariance of $G$-theory \cite[Thm.~II.6.5]{W} as $\sX'_1 =\mathbb A^1_{\sX_1}$. In order to prove Theorem~\ref{thm2} we have
to show that 
\[
\sigma^*: \prolim{t\mapsto \pi t} K_0(A\langle  t\rangle)\to K_0(A)
\]
is a pro-monomorphism. According to Proposition~\ref{prop.k0inj} we find $n>0$ such that
$K_0(\sX')\to K_0(\sX'_n)$ is injective. So by a diagram chase it suffices to show that 
\[
\sigma: \prolim{t\mapsto \pi t} K_0(\sX'_n) \to K_0(\sX_n)
\]
is a pro-monomorphism, which is clear as the morphism $\sX'_n \xrightarrow{t\mapsto \pi^n t}\sX'_n$
factors through $\sX_n$.
\end{proof}

\end{document}